\documentclass[11pt,a4paper]{article}
\usepackage[T1]{fontenc}
\usepackage[utf8]{inputenc}
\usepackage{lmodern}
\usepackage{amsmath,amssymb,amsthm,mathtools}
\usepackage{booktabs,tabularx,array}
\usepackage{microtype,xurl}
\usepackage[margin=0.88in]{geometry}
\usepackage{xcolor}
\usepackage{cite}
\usepackage{caption}
\usepackage{needspace}
\definecolor{linkblue}{RGB}{26,52,108}
\usepackage[colorlinks=true,linkcolor=linkblue,citecolor=linkblue,urlcolor=linkblue]{hyperref}
\usepackage[nameinlink,noabbrev]{cleveref}
\hypersetup{pdftitle={Geodesically Convex Optimization in Hyperbolic Space: Matching Lower Bounds},pdfauthor={},pdfsubject={Unrestricted deterministic first-order oracle complexity},pdfkeywords={geodesic convexity, oracle complexity, lower bounds, upper bounds, hyperbolic geometry}}
\allowdisplaybreaks[1]
\newtheorem{theorem}{Theorem}[section]
\newtheorem{lemma}[theorem]{Lemma}
\newtheorem{proposition}[theorem]{Proposition}
\newtheorem{corollary}[theorem]{Corollary}
\theoremstyle{definition}

\theoremstyle{remark}
\newtheorem{remark}[theorem]{Remark}
\DeclareMathOperator{\dist}{dist}

\DeclareMathOperator{\arcosh}{arcosh}

\DeclareMathOperator{\gconv}{gconv}
\DeclareMathOperator{\spanop}{span}
\newcommand{\HH}{\mathbb H}
\newcommand{\RR}{\mathbb R}
\newcommand{\eps}{\varepsilon}
\newcommand{\ip}[2]{\langle #1,#2\rangle}
\newcommand{\ipL}[2]{\langle #1,#2\rangle_{\mathrm L}}
\newcommand{\norm}[1]{\lVert #1\rVert}
\newcommand{\B}{\overline B}
\newcolumntype{Y}{>{\raggedright\arraybackslash}X}
\title{\textbf{Geodesically Convex Optimization in Hyperbolic Space: Matching Lower Bounds}}
\newcommand{\authornames}{Wendao Wu, Haihan Zhang, Chenheng Zhang, Yanyi Li, Chunyuan Zheng, Cong Fang, Haoxuan Li, Zhouchen Lin}
\newcommand{\affiliation}{Peking University}
\newcommand{\contactemail}{colombo.2026@outlook.com}
\newcommand{\mail}[1]{\href{mailto:#1}{\texttt{#1}}}
\newif\ifanonymous
\anonymousfalse
\ifanonymous
\author{Anonymous Authors}
\hypersetup{pdfauthor={Anonymous Authors}}
\else
\author{Wendao Wu$^{1,*}$\quad Haihan Zhang$^{1,*}$\quad Chenheng Zhang$^{1,*}$\\[0.20em]
Yanyi Li$^{1}$\quad Chunyuan Zheng$^1$\\[0.20em]
Cong Fang$^{1,\dagger}$\quad Haoxuan Li$^{1,\dagger}$\quad Zhouchen Lin$^{1,\dagger}$\\[0.60em]
\small $^1$\affiliation\\[0.35em]
\footnotesize\mail{\contactemail}\\[0.35em]
\small $^*$Equal contribution.\quad $^\dagger$Corresponding authors.}
\hypersetup{pdfauthor={\authornames}}
\fi
\date{}
\begin{document}
\maketitle
\begin{abstract}
We establish matching lower bounds for deterministic first-order optimization of globally Lipschitz geodesically convex functions on hyperbolic space. A global minimizer lies within distance $r$ of a known point, each query returns an exact function value and a selected Riemannian subgradient, and the target objective gap is $\eps Lr$, where $L$ is the Lipschitz constant. For sectional curvature $-k^2<0$, let $\rho=kr$ and $\zeta=\rho/\tanh\rho$. The optimal worst-case query complexity $Q^\star$, taken uniformly over dimensions, satisfies
\[
 Q^\star(\rho,\eps)=\Theta(\zeta\eps^{-2}),
 \qquad \rho>0,\quad 0<\eps\le1/64.
\]
The new lower bound covers arbitrary adaptive deterministic queries and arbitrary outputs, matching the projected-subgradient upper rate of Zhang and Sra~\cite{ZS16}. Building on the nested hyperbolic geometry of Criscitiello and Boumal~\cite[Appendix~A.2]{CB23}, it removes the geodesic-span and historical-halfspace restrictions of their matching lower bound and resolves their conjecture for the Lipschitz problem. The construction extends a radial source from a convex carrier with boundary. A small residual cut selects minimizer candidates while an exact hyperbolic persistence inequality keeps every historical contact valid under future carrier growth. One fixed globally convex objective and one fixed subgradient selection consequently realize all exact replies simultaneously. Each query consumes $O(\eps^2)$ of a geometric potential whose available budget is $\Theta(\zeta)$, yielding the product lower bound in a dimension linear in the query horizon.
\end{abstract}
\noindent\textbf{Keywords:} geodesic convexity; oracle complexity; lower bounds; upper bounds.

\begin{samepage}
\noindent\textbf{AI Usage.}
Nearly the entire research pipeline for this paper was carried out by Colombo, our laboratory’s internal auto-research system, powered by GPT-5.6 Sol. The system also conducted a Lean-backed article audit of the resulting manuscript. The authors subsequently reviewed and approved the mathematical claims, presentation, and formal artifacts, and take responsibility for the final manuscript. The complete Lean audit report and the system’s technical report will be made public at a later date.  \par
\end{samepage}

\paragraph{Publication Intent and Recognition of Independent Work.}

This preprint is released as a research output of our study of
automated scientific research, for public scrutiny and discussion.
As we do not intend to submit this manuscript to a journal or
conference, it has not undergone the full editorial refinement
and internal review we would normally require for submission.
The authors remain responsible for its contents and for correcting
any errors or inaccurate attribution.

We recognize that releasing these results may affect researchers
who have independently been pursuing the same questions, including
work that is not yet public.
If this applies to your work, we invite you to contact us at
\texttt{colombo.2026@outlook.com}.
A brief description of the overlap, together with any dated
materials you are comfortable sharing, would help us understand
its scope and chronology.
Such materials might include draft excerpts, presentation slides,
or revision records; a complete manuscript or unpublished proof
is not required to begin the conversation.

We will review these cases carefully and update citations,
acknowledgments, and accounts of independent or concurrent work
as appropriate, including revising priority claims where warranted.
Nonpublic materials will be treated confidentially and used only
to assess the overlap unless their authors authorize otherwise.
We will obtain explicit permission from the authors before
publicly discussing such materials or providing them to an
automated research system.

\section{Introduction}\label{sec:introduction}

Consider a globally $L$-Lipschitz geodesically convex function
$f:\HH^d\to\RR$ on hyperbolic space of sectional curvature $-1$.
A global minimizer is known to lie in the closed ball $\B(o,R)$ around
a specified point $o$. Given exact function values and selected Riemannian
subgradients, the task is to return a point $\widehat x$ such that
\[
 f(\widehat x)-\min f\le\eps LR,
 \qquad 0<\eps<1.
\]
The algorithm may choose any point of the ambient space at each query.
We ask for the optimal worst-case number of queries among deterministic
algorithms, uniformly over dimensions. This formulation isolates the
information required to optimize from the arithmetic cost of manipulating
the known geometry.

The relevant geometric parameter is $\zeta_R=R/\tanh R$. For curvature
$-k^2$ and physical localization radius $r$, metric rescaling gives $R=kr$.
Projected Riemannian subgradient descent achieves
$O(\zeta_R\eps^{-2})$ queries~\cite{ZS16}. Its curvature factor approaches
one as $R$ tends to zero and grows proportionally to $R$ at large radius.
The resulting penalty could reflect either a limitation of the method
or an unavoidable information cost. Distinguishing these possibilities
requires a lower bound for arbitrary adaptive algorithms.

Criscitiello and Boumal~\cite{CB23} formulate the dimension-uniform
nonsmooth problem as P2. On a shared high-accuracy regime, their general
lower bounds give
\begin{equation}\label{eq:prior-benchmark}
 \Omega\!\left(\frac{\zeta_R}{\log(64\zeta_R)}
          +\frac{1}{\zeta_R^2\eps^2}\right).
\end{equation}
They also establish $\Omega(\zeta_R\eps^{-2})$ for algorithms satisfying
geodesic-span and historical-halfspace restrictions. The unrestricted
benchmark in \Cref{eq:prior-benchmark} leaves a polynomial gap in the joint
dependence on curvature and accuracy. Its curvature and accuracy terms
arise from separate constructions.

Criscitiello and Boumal~\cite{CB23} conjectured the multiplicative
curvature--accuracy rate and laid the mathematical foundation for the
present result. Their formulation of P2 and their matching lower bound
under the above restrictions identify the problem and the target rate.
Their nested hyperbolic construction~\cite[Appendix~A.2]{CB23}, including
the candidate spheres, recentering identities, and slow decay of the
uncertainty radii, supplies the geometric foundation of our proof and
the mechanism producing the multiplicative curvature factor. Our
AI-assisted work builds on these contributions to resolve the
unrestricted deterministic problem through an evolving convex carrier
and contact-persistence certificates.

\paragraph{Our contribution.}
We prove the unrestricted product lower bound. If $Q^\star(R,\eps)$ is
the dimension-uniform deterministic query complexity, then
\Cref{thm:main} establishes
\[
 Q^\star(R,\eps)=\Theta(\zeta_R\eps^{-2})
 \quad\text{for every }R>0\text{ and }0<\eps\le1/64,
\]
with universal constants. Thus the curvature penalty must multiply the
entire nonsmooth accuracy cost. The lower bound holds for a fixed
single-valued oracle returning exact values and full selected
subgradients; it permits unqueried final outputs. The matching upper
rate is inherited from projected subgradient descent. We give a
self-contained proof of that benchmark to make the oracle and output
conventions explicit.

The central technical step added here is a simultaneous global completion
of the oracle transcript on this existing geometric foundation. We
combine the nested candidate spheres of Criscitiello and Boumal with a
convex carrier with boundary. Every new query creates a supporting
contact; whenever needed, a small residual cut ensures that its normal
remains valid under all future updates. Retaining the convex hull of all historical public balls
then preserves every exact reply for one final objective. The geometric
budget is $\log\sinh S$, where $S$ is the remaining candidate radius:
each query spends at most $O(\eps^2)$, while halving the original radius
requires $\Theta(\zeta_R)$. \Cref{sec:overview} explains this mechanism
before the detailed construction.

\begin{table}[t]
\caption{Deterministic exact first-order query bounds for objective gap $\eps LR$, with $\zeta_R=R/\tanh R$. Dimension-uniform and dimension-dependent guarantees are listed separately.}
\label{tab:comparison}
\centering
\small
\setlength{\tabcolsep}{5pt}
\renewcommand{\arraystretch}{1.18}
\begin{tabularx}{\textwidth}{@{}>{\raggedright\arraybackslash}p{0.225\textwidth}>{\raggedright\arraybackslash}p{0.365\textwidth}Y@{}}
\toprule
\textbf{Work / bound type} & \textbf{Query complexity} & \textbf{Algorithms and range}\\
\midrule
\multicolumn{3}{@{}l}{\textbf{Dimension-uniform bounds}}\\[5pt]
Zhang--Sra~\cite{ZS16}\newline
\textit{Upper bound}
& $O(\zeta_R\eps^{-2})$
& Projected subgradient descent;\newline
  $R>0$, $0<\eps<1$.\\[7pt]
Criscitiello--Boumal~\cite{CB23}\newline
\textit{Lower bound}
& $\displaystyle\Omega\!\left(\frac{\zeta_R}{\log(64\zeta_R)}
                       +\frac{1}{\zeta_R^2\eps^2}\right)$
& Unrestricted queries;\newline
  $R\ge64$, $0<\eps\le\dfrac{1}{2^{18}\log^2\zeta_R}$.\\[7pt]
Criscitiello--Boumal\newline
\cite[Theorem~29]{CB23}\newline
\textit{Lower bound}
& $\Omega(\zeta_R\eps^{-2})$
& Geodesic-span and historical-halfspace restrictions (A1/A2);\newline
  $0<\eps\le1/(4\sqrt2)$.\\[7pt]
\textbf{This work}\newline
(\Cref{thm:main})\newline
\textit{New lower bound}
& \boldmath$\Omega(\zeta_R\eps^{-2})$
& \textbf{Unrestricted deterministic};\newline
  $R>0$, $0<\eps\le1/64$.\\[3pt]
\midrule
\multicolumn{3}{@{}l}{\textbf{Dimension-dependent bound}}\\[5pt]
Zhang et al.~\cite{Klein26}\newline
\textit{Upper bound}
& $O\!\left(d^2[R+\log(e/\eps)]\right)$
& Klein-model cutting planes;\newline
  $R>0$, $0<\eps<1$.\\
\bottomrule
\end{tabularx}
\par\vspace{5pt}
\begin{minipage}{\textwidth}
\footnotesize
\textbf{Matching rate.} The new lower bound and the Zhang--Sra upper bound give
$Q^\star(R,\eps)=\Theta(\zeta_R\eps^{-2})$ for $R>0$ and $0<\eps\le1/64$.
\par\vspace{2pt}
\textit{Scope.} Lower-bound dimensions may grow with the query horizon.
The accuracy term $\Omega(\zeta_R^{-2}\eps^{-2})$ in the unrestricted
Criscitiello--Boumal bound holds for every $R>0$ when nontrivial.
\end{minipage}
\end{table}

\paragraph{Dimension and scope.}
\Cref{tab:comparison} separates dimension-uniform and dimension-dependent
guarantees. The lower construction uses $d=T+2$ at horizon
$T=\Theta(\zeta_R\eps^{-2})$. In a fixed dimension, the Klein-model
cutting-plane method of Zhang et al.~\cite{Klein26} can have a much smaller
accuracy dependence. This is compatible with a worst-case guarantee
uniform over all dimensions. Our theorem addresses the deterministic
nonsmooth problem; its accuracy range and the distinct questions of
randomization and dimension dependence are discussed in \Cref{sec:scope}.

\paragraph{Organization.}
\Cref{sec:related} discusses the surrounding literature.
\Cref{sec:model} defines the oracle model and states the main theorem.
After the overview in \Cref{sec:overview}, \Cref{sec:basic} develops the
metric extension and its contact certificates,
\Cref{sec:convex-domain-geometry} proves their persistence, and
\Cref{sec:oracle} constructs the resisting oracle and proves the lower
bound. \Cref{sec:upper} supplies the matching upper bound.

\section{Related Work}\label{sec:related}

\paragraph{Riemannian first-order methods.}
Ferreira and Oliveira~\cite{FO98} introduced a Riemannian subgradient
algorithm and established convergence under nonnegative sectional
curvature. Zhang and Sra~\cite{ZS16} developed nonasymptotic first-order
analyses that expose curvature through comparison geometry. For
Lipschitz geodesically convex objectives, their subgradient estimate
scales as the square root of the ratio between the curvature factor and
the iteration count, giving the upper benchmark used here after converting
diameter to localization radius. Ferreira, Louzeiro, and
Prudente~\cite{FLP19} further analyzed exogenous and Polyak step sizes on
manifolds with lower-bounded curvature. These analyses quantify curvature
in subgradient progress estimates. Our lower bound establishes the same
product dependence for every deterministic algorithm, uniformly over dimensions.

\paragraph{Curvature lower bounds and convex interpolation.}
Hamilton and Moitra~\cite{HM21} established an obstruction to robust
acceleration in the hyperbolic plane with noisy first-order access.
Criscitiello and Boumal~\cite{CB22} proved an obstruction for smooth,
strongly geodesically convex optimization with exact oracles and
unrestricted queries. Their later nonsmooth work~\cite{CB23} is the
closest predecessor to the present theorem. Its A1 restriction constrains
the geodesic span of queries, and A2 requires queries to remain in the
historical supporting halfspaces. Removing both requires exact global
consistency outside those regions. A further obstruction is that
pairwise value--subgradient support inequalities need not admit a
geodesically convex interpolant, even for three points
\cite[Proposition~16]{CB23}. We construct the final function explicitly,
so consistency follows from a common carrier and common supporting
normals. The convexity principles behind metric infimal convolution
are classical: Azagra and Ferrera~\cite{AF06} study partial minimization
of jointly convex functions and distance-based regularization on
nonpositively curved manifolds. Our contribution is the choice and
evolution of the carrier that preserve all exact contacts.

\paragraph{Independent concurrent work.}
Criscitiello~\cite{C26} independently obtains the same dimension-uniform
deterministic complexity $\Theta(\zeta_R\eps^{-2})$ in a universal
small-accuracy regime. Both works resolve the unrestricted Lipschitz
case of the conjecture in Criscitiello and Boumal~\cite{CB23}, and both
use the nested hyperbolic geometry developed in their Appendix~A.2.
The hard functions and the mechanisms for handling arbitrary queries
differ. Criscitiello uses reflected distance-to-halfspace ramps with
an advance-or-stall argument, followed by a black-box reduction removing
the geodesic-span assumption. We instead extend a radial source from
a growing convex carrier and preserve all historical contacts through
supporting halfspaces, constructing a resisting oracle directly for
arbitrary queries. Thus the two independent solutions share the
foundational geometry of the 2023 work while supplying different
mechanisms for overcoming its algorithmic restrictions.

\paragraph{Dimension-dependent cutting planes.}
Rusciano~\cite{Rusciano20} established a Riemannian centerpoint theorem
and its consequences for cutting-plane optimization. Such oracle
guarantees raise a separate issue: computing the centerpoint may be
expensive. Criscitiello, Mart\'inez-Rubio, and Boumal~\cite{CMRB23} ask
for $O(\operatorname{poly}(\zeta_R,d)\log(1/\eps))$ queries with
$\operatorname{poly}(\zeta_R,d)$ arithmetic work per iteration on their
class of Riemannian manifolds. In hyperbolic space, Zhang, Yang, Zhu,
and Zhang~\cite[Theorem~1.1, version~4]{Klein26} obtain the explicit
bound
\[
 \left\lceil 2d(d+1)
 \log\!\left(\frac{16\sinh R\cosh R}{R\eps}\right)\right\rceil
 =O\!\left(d^2[R+\log(e/\eps)]\right).
\]
Their method represents supporting halfspaces in the Klein model. Its
explicit dependence on $d$ permits logarithmic accuracy dependence in
fixed dimension. Our lower bound concerns the supremum over dimensions
and imposes no arithmetic budget.

\paragraph{Horospherical structure and stronger oracles.}
Lewis, L\'opez-Acedo, and Nicolae~\cite{LLN24} study a subgradient-type
method using horospherically convex sublevel sets and supporting
horoballs; their framework also allows quasiconvex objectives.
Criscitiello and Kim~\cite{CK25} obtain curvature-independent rates for
h-convex objectives and averages of h-convex components supplied with
component h-subgradients. Goodwin, Lewis, L\'opez-Acedo, and
Nicolae~\cite{GLLN26} develop stochastic and incremental methods using
Busemann subgradients on Hadamard spaces. These formulations exploit
support geometry that is not available from an arbitrary ordinary
Riemannian subgradient of a general geodesically convex function.
Their curvature-independent guarantees thus do not contradict the
present lower bound. At higher derivative orders, Han~\cite{Han26}
compares horospherical and geodesic convexity for smooth strongly convex
objectives with exact $p$th-order access, $p\ge2$. These additional
derivatives and regularity assumptions go beyond the present oracle model.

\section{Oracle Model and Main Result}\label{sec:model}

We first work in $\HH^d$ with sectional curvature $-1$. The reference
point $o$, dimension $d\ge2$, localization radius $R>0$, Lipschitz constant
$L>0$, and accuracy $\eps>0$ are known to the algorithm. We write
$\dist$ for geodesic distance, $\exp_x$ and $\log_x$ for the exponential
and logarithm maps, and $\ip{\cdot}{\cdot}$ for the Riemannian inner product
at the indicated base point. Scalar logarithms are natural.

\subsection{Functions, oracle replies, and charged queries}

Let $\mathcal F_d(R,L)$ be the class of functions $f:\HH^d\to\RR$ such
that $f$ is geodesically convex on the whole space,
\[
 |f(x)-f(y)|\le L\dist(x,y)\quad\text{for all }x,y\in\HH^d,
 \qquad \operatorname*{argmin}f\cap\B(o,R)\ne\varnothing.
\]
Here geodesic convexity means that
$f(\gamma(t))\le(1-t)f(\gamma(0))+tf(\gamma(1))$ for every
constant-speed geodesic segment $\gamma:[0,1]\to\HH^d$.
For any $x$, define its global Riemannian subdifferential by
\[
 \partial f(x)=\left\{g\in T_x\HH^d:
 f(y)\ge f(x)+\ip{g}{\log_x y}\ \text{for all }y\in\HH^d\right\}.
\]
\Cref{lem:selection} verifies that this set is nonempty for every
function in the class. A legal exact first-order oracle is a fixed map
$O_f:x\mapsto(f(x),g_f(x))$ with $g_f(x)\in\partial f(x)$.
The algorithm must succeed for every legal selection $g_f$.

One call returns both the value and the full selected tangent vector at
one arbitrary point of $\HH^d$. Queries need not lie in the localization
ball, in a span determined by earlier replies, or in earlier supporting
halfspaces. The algorithm is deterministic and may perform arbitrary
computation between calls; exact geometric operations, including
exponential maps, logarithm maps, and projections, incur no objective-oracle
calls. Its output may be an unqueried point. The required guarantee is
\begin{equation}\label{eq:target}
 f(\widehat x)-\min f\le\eps LR.
\end{equation}

Let $\mathfrak O(f)$ denote all legal fixed oracle maps for $f$. For a
deterministic algorithm $A$, let $N_A(f,O_f)$ count its calls before
returning its output $\widehat x_A(f,O_f)$, with $N_A=+\infty$ if it never
returns. Define
\begin{equation}\label{eq:minimax-definition}
 Q_d^\star(R,\eps)=\inf\left\{N\in\mathbb N_0:
 \begin{array}{l}
 \exists A\ \forall f\in\mathcal F_d(R,L)\ \forall O_f\in\mathfrak O(f),\\
 N_A(f,O_f)\le N,\quad
 f(\widehat x_A(f,O_f))-\min f\le\eps LR
 \end{array}\right\}.
\end{equation}
Multiplying or dividing every returned value and vector by $L$ gives
exact simulations between this class and its $L=1$ normalization, so
the complexity does not depend on $L$. We suppress that parameter and
put $Q^\star(R,\eps)=\sup_{d\ge2}Q_d^\star(R,\eps)$.
The infimum of an empty feasible set is $+\infty$. Restricting attention
to fixed oracle maps does not weaken our lower construction: every
adversarial transcript below has such a completion. The projected
subgradient upper bound also allows any legal selected subgradient at
each call.

\Needspace{22\baselineskip}
\subsection{Matching complexity}

\begin{theorem}[Matching deterministic complexity]\label[theorem]{thm:main}
For every $R,L>0$ and $0<\eps\le1/64$,
\begin{equation}\label{eq:main-bounds}
 \boxed{\quad
 \frac{\zeta_R}{260\eps^2}
 \le Q^\star(R,\eps)
 \le \left\lceil\frac{3\zeta_R}{\eps^2}\right\rceil,
 \qquad \zeta_R=\frac{R}{\tanh R}.
 \quad}
\end{equation}
More precisely, put
\begin{equation}\label{eq:main-parameters}
 \mu=4\eps,\qquad
 \eta=\frac{2\mu}{1-\mu^2},\qquad
 \sigma=\sqrt{1-\eta^2},\qquad
 T=\left\lfloor
 \frac{\log(2\cosh(R/2))}{-\log\sigma}
 \right\rfloor.
\end{equation}
For every deterministic algorithm returning an output after at most $T$ queries in dimension
$d=T+2$, there is one $f\in\mathcal F_d(R,L)$, with a unique global
minimizer, and one legal fixed first-order oracle for $f$, such that its
output satisfies
\[
 f(\widehat x)-\min f\ge2\eps LR.
\]
The witnessing objective and oracle are fixed for the entire run of the
algorithm.
\end{theorem}

The construction proving the lower bound appears in \Cref{sec:oracle},
using the contact-persistence results of
\Cref{sec:convex-domain-geometry}. The upper bound is
\Cref{prop:upper}. Together they show that the known curvature penalty
is unavoidable throughout the stated accuracy range.

\begin{corollary}[Strict localization]\label[corollary]{cor:open-ball}
If the minimizer is required to lie in the open ball $B(o,R)$, the same
dimension-uniform rate holds. In particular, for $0<\eps\le1/64$ its
complexity is at least $3\zeta_R/(1040\eps^2)$ and at most
$\lceil3\zeta_R/\eps^2\rceil$.
\end{corollary}
\begin{proof}
Apply the more precise lower-bound assertion in \cref{thm:main} at
$R_0=3R/4$, with the same $\eps$ and $L$. Its output gap is at least
$2\eps LR_0=\tfrac32\eps LR>\eps LR$, while its minimizer has distance
$R_0<R$ from $o$. Also $\zeta_{R_0}\ge\tfrac34\zeta_R$.
The closed-ball upper bound applies to this smaller function class.
\end{proof}

\subsection{Metric scaling and parameter interpretation}

For a space of curvature $-k^2<0$ with metric tensor $\mathsf g$, replace
the metric by $\mathsf g'=k^2\mathsf g$. The new curvature is $-1$ and
distances multiply by $k$. A physical radius $r$ and Lipschitz constant
$L$ become $R=kr$ and $L'=L/k$, so
$\eps L'R=\eps Lr$. A selected subgradient vector $h$ becomes
$h'=h/k^2$. The connection and the tangent-vector exponential and
logarithm maps are unchanged; hence
$\ip{h'}{\log_x y}_{\mathsf g'}=\ip{h}{\log_x y}_{\mathsf g}$,
and the complete oracle support inequality is preserved. This gives
the dependence on $\rho=kr$ stated in the abstract, with
$\zeta=\rho/\tanh\rho$.

All asymptotic constants in the main theorem are universal over its
stated parameters. In particular, $\zeta_R\to1$ as $R\to0$ and
$\zeta_R\sim R$ as $R\to\infty$. The lower bound therefore captures
the transition from nearly Euclidean local geometry to large
hyperbolic localization radii without fixing the curvature factor.

\section{Proof Overview}\label{sec:overview}

The proof of \Cref{thm:main} must control both the information in a new
query and the global compatibility of all previous replies. A selected
subgradient is a complete exact vector, and a small component in a hidden
direction cannot be treated as unobservable. Moreover, checking support
inequalities only at queried points does not construct a global convex
objective. We preserve the transcript through supporting halfspaces of
one explicitly defined carrier.

\paragraph{A convex carrier turns contacts into full oracle replies.}
For a candidate minimizer $p$ and a compact geodesically convex set $C$
containing it, consider the normalized objective
\[
 F_{C,p}(y)=\min_{x\in C}\{\mu\dist(x,p)+\dist(y,x)\},
 \qquad \mu=4\eps.
\]
\Cref{lem:metric-carrier} proves that this metric extension is globally
convex, globally $1$-Lipschitz, and bounded below by
$\mu\dist(y,p)$. Thus separation from $p$ implies objective error.
\Cref{lem:contact-realization} shows that a supporting normal at a contact
$z\in C$ fixes both the exact value and a full global subgradient at the
query. Preserving that normal under enlargement of $C$ preserves the reply.
The boundary of $C$ matters: even after a query direction becomes public,
an old exterior query remains separated from the carrier by its original
supporting halfspace.

\paragraph{One residual cut pays for an arbitrary query.}
The public totally geodesic subspace $V$ and the orthogonal sphere of
candidate minimizers follow the construction of Criscitiello and
Boumal~\cite[Appendix~A.2]{CB23}. We augment this geometry with a compact
convex public domain $P\subset V$. The common-projection property of
the candidate sphere~\cite[Lemma~26]{CB23} makes all candidates'
distances to any point of $P$ coincide. We can
therefore choose a candidate-independent contact by optimizing over $P$.
When the contact direction has a nonzero residual component in direction
$e$, we retain a latitude
$\ip{e}{v}=-\eta$ of the residual unit sphere, where
$\eta=2\mu/(1-\mu^2)$. For an on-domain query, this residual direction
is chosen explicitly. This enlarges the public subspace by at most one dimension.
The contact angle is allowed to vary with the query, which makes the
same small latitude feasible for all contacts.

The public domain contains an intrinsic ball whose radius $r_S$ satisfies
$\sinh r_S=\eta\tanh S$, where $S$ is the candidate-sphere radius.
This ball supplies a quantitative reserve in every public supporting
normal. \Cref{prop:cd-small-latitude} combines that reserve with a scalar
identity and a branch argument to prove that the chosen contact will
remain valid at every smaller residual radius. The Lorentz calculation
in \Cref{prop:cd-future-balls} converts this scalar statement into
containment of every future public ball.

\paragraph{The entire transcript has one global completion.}
At each update, $P$ becomes the convex hull of its previous value and
the new public ball. These balls are generally not nested, so retaining
the entire hull is necessary for the archive invariant. Contact
optimality supports the old hull, while future-ball persistence supports
every later insertion. For each surviving $p$, the carrier
$C=\overline{\gconv}(P\cup\{p\})$ then satisfies all historical normal
inequalities simultaneously. \Cref{prop:prefix-consistency} proves that
the associated $F_{C,p}$ reproduces every value and every full selected
subgradient after every finite prefix.

\paragraph{The geometric budget.}
The recentering and radius budget are those of Criscitiello and
Boumal~\cite[Eq.~(7), Lemma~24]{CB23}. In our notation, recentering
the surviving latitude gives the exact update
\[
 \sinh S^+=\sigma\sinh S,
 \qquad \sigma=\sqrt{1-\eta^2}.
\]
For $\mu=4\eps$, its potential cost is
$-\log\sigma=\Theta(\eps^2)$. The available decrease before $S$ drops
from $R$ to $R/2$ is
\[
 \log\frac{\sinh R}{\sinh(R/2)}
 =\log\bigl(2\cosh(R/2)\bigr)=\Theta(\zeta_R).
\]
Thus a horizon of order $\zeta_R\eps^{-2}$ leaves a candidate sphere of
radius at least $R/2$. A surviving candidate can be chosen at distance
at least $R/2$ from any final output, and the metric-extension floor
gives an objective gap of at least $2\eps R$. Finally, deterministic
replay turns the adaptive construction into a run against one fixed
function and one fixed oracle. \Cref{sec:oracle} makes the dimension,
termination alternatives, and uniform constants explicit.

\section{Convex Carriers and Exact Oracle Replies}\label{sec:basic}

The geometric preliminaries below make the global extension and its
selected subgradients explicit. We first record the distance facts needed
for convexity and global support, then identify the supporting-contact
condition that survives the later construction.

We use the hyperboloid
\[
 \HH^d=\{x\in\RR^{d+1}:\ipL{x}{x}=-1,\ x_0>0\},\qquad
 \ipL{x}{y}=-x_0y_0+\sum_{j=1}^d x_jy_j.
\]
If $v\in T_x\HH^d$ is unit, then
\[
 \exp_x(tv)=\cosh t\,x+\sinh t\,v,
 \qquad \cosh\dist(x,y)=-\ipL{x}{y}.
\]
For $D=\dist(x,p)>0$, the unit direction to $p$ and the radial gradient are
\begin{equation}\label{eq:radial-unit}
 u_{x,p}=\frac{p-\cosh D\,x}{\sinh D},
 \qquad \nabla_x\dist(x,p)=-u_{x,p}.
\end{equation}
A public totally geodesic subspace has the form $V=\HH^d\cap W$, where
$W$ is a Lorentz subspace containing a timelike vector. Its orthogonal
complement $W^{\perp_{\mathrm L}}$ is Euclidean. We allow $V=\{o\}$.
Closed geodesic convex hulls of bounded subsets are compact in this
finite-dimensional space. For example, the Klein model maps them to closed
Euclidean convex hulls inside a compact subset of the unit ball.

For a closed convex $P\subset V$ and $z\in P$, write
\[
 N_P^V(z)=\{n\in T_zV:\ip{n}{\log_z x}\le0\text{ for all }x\in P\}.
\]
For $n\in T_z\HH^d$, the sign of $\ip{n}{\log_z x}$ equals the sign
of $\ipL{n}{x}$, since
\begin{equation}\label{eq:sign}
 \ip{n}{\log_z x}
 =\frac{\dist(z,x)}{\sinh\dist(z,x)}\ipL{n}{x}
 \quad(x\ne z).
\end{equation}
Thus a normal defines a geodesically convex halfspace through $z$.

\begin{lemma}[Global distance facts]\label[lemma]{lem:distance-facts}
Distance on $\HH^d$ is jointly convex along pairs of constant-speed
geodesics. For every $z$, the map $\log_z:\HH^d\to T_z\HH^d$ is
$1$-Lipschitz.
\end{lemma}
\begin{proof}
For the first assertion, let the two endpoints move along geodesics.
At a time when their distance $\ell$ is positive, parametrize the
reference joining segment by arc length and use the variation field
of the affinely parametrized joining segments. The second
variation of length gives
\[
 \ell''=\int_0^\ell
 \left(\norm{D_sJ^\perp}^2
       -\ip{\mathrm{Rm}(J^\perp,T)T}{J^\perp}\right)\,ds\ge0.
\]
Here $T$ is the unit tangent and $J^\perp$ is the perpendicular component
of the variational field. Endpoint acceleration terms vanish because the
endpoint curves are geodesics; sectional curvature $-1$ makes the
curvature contribution nonnegative. At an endpoint coincidence the
distance is continuous and nonnegative. The one-sided slopes of its
adjacent convex restrictions therefore satisfy
$\ell'_-\le0\le\ell'_+$. These slopes are finite because the sum of the
endpoint speeds bounds the Lipschitz constant of their distance. This
extends convexity through such points; identical endpoint curves give
identically zero distance.

For the second assertion, compare the hyperbolic polar metric
$dr^2+\sinh^2r\,d\theta^2$ with the Euclidean tangent-space metric
$dr^2+r^2d\theta^2$. The differential of $\log_z$ cannot increase length.
Integrating this inequality along a geodesic between any two points gives
the claimed global Lipschitz inequality.
\end{proof}

\begin{lemma}[Existence of a fixed subgradient selection]\label[lemma]{lem:selection}
If $f:\HH^d\to\RR$ is globally geodesically convex and globally
$L$-Lipschitz, then $\partial f(x)$ is a nonempty compact convex subset
of the closed radius-$L$ ball in $T_x\HH^d$ for every $x$. It therefore
has a unique minimum-norm element.
\end{lemma}
\begin{proof}
At a differentiability point, the one-dimensional convex slope inequality
along each geodesic gives
$f(y)\ge f(x)+\ip{\nabla f(x)}{\log_x y}$ for every $y$.
Lipschitz continuity implies $\norm{\nabla f(x)}\le L$.
Rademacher's theorem in local coordinate charts supplies differentiability
points $x_j\to x$. Parallel transport their gradients to $T_x\HH^d$;
a subsequence converges to some $v$ with $\norm v\le L$. Continuity of
the logarithm map, parallel transport, and $f$ lets us pass each global
support inequality to the limit, giving
$f(y)\ge f(x)+\ip{v}{\log_x y}$ for every $y$.
Thus $\partial f(x)$ is nonempty. Its defining inequalities make it an
intersection of closed affine halfspaces in $T_x\HH^d$, hence closed and
convex. For any $g\in\partial f(x)$, apply its support inequality at
$\exp_x(tg/\norm g)$ and compare with Lipschitz continuity to obtain
$\norm g\le L$ (the case $g=0$ is immediate).
Finite dimensionality gives compactness, and strict convexity of the
squared norm gives the unique minimum-norm element.
\end{proof}

We now extend a radial source from a compact convex carrier to the whole
ambient space. The joint-convexity principle underlying this construction
also appears in the infimal-convolution framework of Azagra and
Ferrera~\cite[Lemma~3.1 and Corollary~4.2]{AF06}. The following contact
lemma shows how to preserve complete oracle replies as the carrier grows.

\begin{lemma}[A bounded convex carrier]\label[lemma]{lem:metric-carrier}
Let $0<\mu<1$, let $K\subset\HH^d$ be nonempty, compact, and geodesically
convex, and let $p\in K$. Define
\begin{equation}\label{eq:final-function}
 F_{K,p}(y)=\min_{x\in K}
 \{\mu\dist(x,p)+\dist(y,x)\}.
\end{equation}
Then $F_{K,p}$ is globally geodesically convex, globally $1$-Lipschitz,
and satisfies
\begin{equation}\label{eq:floor}
 F_{K,p}(y)\ge\mu\dist(y,p),\qquad F_{K,p}(p)=0.
\end{equation}
Consequently $p$ is its unique global minimizer. On $K$,
$F_{K,p}(x)=\mu\dist(x,p)$.
\end{lemma}
\begin{proof}
Compactness gives attainment. Given two evaluation points and their
minimizing contacts, interpolate both pairs along geodesics. The interpolated
contact remains in $K$, and joint convexity of distance proves convexity of
$F_{K,p}$. The infimum of these $1$-Lipschitz functions is $1$-Lipschitz.
For every $x\in K$,
\[
 \mu\dist(x,p)+\dist(y,x)
 \ge\mu\dist(y,p)+(1-\mu)\dist(y,x).
\]
Taking the minimum proves the floor; taking $x=y\in K$ proves agreement
with the source. The remaining claims follow immediately.
\end{proof}

\begin{lemma}[Exact preservation of a contact]\label[lemma]{lem:contact-realization}
In \cref{lem:metric-carrier}, let $z\in K$, let $D=\dist(z,p)>0$, and
let $q\in T_z\HH^d$ be unit. Suppose
\begin{equation}\label{eq:normal-contact}
 n=q+\mu u_{z,p}\in N_K(z).
\end{equation}
For any $\ell>0$, set $y=\exp_z(\ell q)$. Then
\begin{equation}\label{eq:external-answer}
 F_{K,p}(y)=\mu D+\ell,
 \qquad
 -\frac{\log_y z}{\ell}\in\partial F_{K,p}(y).
\end{equation}
At $y=z$ one instead has
\begin{equation}\label{eq:internal-answer}
 F_{K,p}(z)=\mu D,\qquad q\in\partial F_{K,p}(z).
\end{equation}
For $\ell>0$, the query $y$ is outside $K$.
\end{lemma}
\begin{proof}
For $x\in K$, radial convexity and the normal inequality give
\begin{equation}\label{eq:source-support}
 \mu\dist(x,p)
 \ge\mu D-\mu\ip{u_{z,p}}{\log_z x}
 \ge\mu D+\ip{q}{\log_z x}.
\end{equation}
If $y=\exp_z(\ell q)$, joint convexity of distance gives, for all $w$ and
$x\in K$,
\[
 \dist(w,x)\ge
 \ell+\left\langle-\frac{\log_yz}{\ell},\log_yw\right\rangle
 -\ip{q}{\log_zx}.
\]
Adding \cref{eq:source-support} and minimizing over $x$ proves the selected
support at $y$. Equality of values follows by using the contact $x=z$.
For $y=z$, log contraction gives
\[
 \ip{q}{\log_zx}+\dist(w,x)\ge\ip{q}{\log_zw},
\]
which proves the selected support at $z$. Finally,
\[
 \ipL{n}{y}=\sinh\ell\,(1+\mu\ip{q}{u_{z,p}})>0,
\]
whereas $\ipL{n}{x}\le0$ for all $x\in K$. Thus $y\notin K$.
\end{proof}

\section{Preserving Historical Contacts}
\label{sec:convex-domain-geometry}

The preceding section reduces exact oracle consistency to supporting
normals of the carrier. We now construct normals that support the whole
past public domain and every future public ball. The current ball yields
a normal reserve; a scalar persistence inequality ensures that this
reserve remains sufficient as the candidate sphere shrinks.

Throughout this section, fix
\begin{equation}\label{eq:cd-parameters}
 0<\mu\le\frac1{16},\qquad
 \tau=\frac{2\mu}{1+\mu^2},\qquad
 \kappa=\frac{1-\mu^2}{1+\mu^2},\qquad
 \eta=\frac{\tau}{\kappa}=\frac{2\mu}{1-\mu^2},\qquad
 a_0=\frac1\kappa.
\end{equation}
Thus \(\tau^2+\kappa^2=1\), \(a_0^2=1+\eta^2\), and \(0<\eta<1\).
For \(S>0\), define the public-ball radius by
\begin{equation}\label{eq:cd-ball-radius}
 \sinh r_S=\eta\tanh S.
\end{equation}

\subsection{Lorentz halfspaces and intrinsic balls}

\begin{lemma}[Halfspace--ball criterion]\label[lemma]{lem:cd-ball-criterion}
Let \(V=\mathbb H^d\cap W\), let \(c\in V\), and let \(h\in W\) be
spacelike and nonzero.  Write \(\alpha=\sqrt{\langle h,h\rangle_L}\).
The Lorentz halfspace
\[
 H_h^-\cap V=\{x\in V:\langle h,x\rangle_L\le0\}
\]
contains the intrinsic closed ball \(\overline B_V(c,r)\) if and only if
\begin{equation}\label{eq:cd-ball-criterion}
 -\langle h,c\rangle_L\ge\alpha\sinh r.
\end{equation}
When \(\langle h,c\rangle_L\le0\), the distance in \(V\) from \(c\) to
the boundary hyperplane is
\[
 \operatorname{arsinh}
 \left(\frac{-\langle h,c\rangle_L}{\alpha}\right).
\]
For \(h=0\), the halfspace is all of \(V\), and containment is automatic.
\end{lemma}

\begin{proof}
Put \(m=h/\alpha\) and \(\lambda=-\langle m,c\rangle_L\).  If
\(\lambda\ge0\), then
\[
 z=\frac{c+\lambda m}{\sqrt{1+\lambda^2}}
\]
belongs to \(V\cap m^{\perp_L}\), and
\(d(c,z)=\operatorname{arsinh}\lambda\).  This is the perpendicular foot
onto the totally geodesic hyperplane.  Equivalently, maximizing
\(\langle h,x\rangle_L\) over points
\(x=\cosh r\,c+\sinh r\,e\), \(e\in T_cV\), \(\|e\|=1\), gives
\[
 \langle h,c\rangle_L\cosh r+
 \sqrt{\alpha^2+\langle h,c\rangle_L^2}\,\sinh r.
\]
This quantity is nonpositive exactly when
\(-\langle h,c\rangle_L\ge\alpha\sinh r\).  If the center itself is
outside the halfspace, containment fails.  These observations also prove
the asserted distance formula.
\end{proof}

\subsection{A scalar persistence inequality}

The halfspace--ball criterion becomes a ratio involving the distance
from a contact to the hidden minimizer. Future recentering increases its
scalar argument $a$ while keeping the physical contact fixed. We need
the ratio to remain above one threshold for every larger argument;
pointwise validity at the current stage would be insufficient.

For an angle parameter \(\vartheta\in(\mu,1]\), put
\begin{equation}\label{eq:cd-scalar-functions}
 A_{\vartheta}=1+\mu^2-2\mu\vartheta,\qquad
 B_{\vartheta}(a)=(\vartheta-\mu)a+\frac\mu a,
 \qquad
 F_{\vartheta}(a)=
 \frac{B_{\vartheta}(a)}
 {\sqrt{A_{\vartheta}-\mu^2/a^2}},\quad a\ge1.
\end{equation}
All denominators are positive: since \(\vartheta\le1\),
\begin{equation}\label{eq:cd-positive-denominator}
 A_{\vartheta}-\frac{\mu^2}{a^2}
 \ge (1-\mu)^2-\mu^2=1-2\mu>0.
\end{equation}

\begin{lemma}[Perfect square at the reference angle]
\label[lemma]{lem:cd-perfect-square}
For every \(a\ge1\),
\begin{equation}\label{eq:cd-perfect-square}
 B_{\tau}(a)^2-
 \eta^2\left(A_{\tau}-\frac{\mu^2}{a^2}\right)
 =\mu^2\left(\kappa a-\frac{1}{\kappa a}\right)^2.
\end{equation}
Consequently \(F_{\tau}(a)\ge\eta\), with equality precisely at
\(a=a_0\).  If \(\vartheta\ge\tau\), then
\(F_{\vartheta}(a)\ge\eta\) for every \(a\ge1\).
\end{lemma}

\begin{proof}
Directly from \eqref{eq:cd-parameters},
\[
 \tau-\mu=\mu\kappa,\qquad
 A_{\tau}=\frac{(1-\mu^2)^2}{1+\mu^2},\qquad
 \eta^2=\frac{4\mu^2}{(1-\mu^2)^2}.
\]
Expanding the left side of \eqref{eq:cd-perfect-square}, the coefficients
of \(a^2\), \(1\), and \(a^{-2}\) are respectively
\(\mu^2\kappa^2\), \(-2\mu^2\), and \(\mu^2/\kappa^2\).
This proves the identity and its equality case.  For fixed \(a\), increasing
\(\vartheta>\mu\) increases the positive numerator \(B_{\vartheta}(a)\)
and decreases its positive squared denominator.  Thus
\(F_{\vartheta}(a)\) is increasing in \(\vartheta\), proving the last claim.
\end{proof}

\begin{lemma}[Persistence on the right branch]\label[lemma]{lem:cd-right-branch}
Suppose \(\mu<\vartheta\le1\), \(a\ge a_0\), and
\(F_{\vartheta}(a)\ge\eta\).  Then
\[
 F_{\vartheta}(a')\ge\eta\qquad\text{for every }a'\ge a.
\]
\end{lemma}

\begin{proof}
The case \(\vartheta\ge\tau\) follows from
\Cref{lem:cd-perfect-square}.  Suppose \(\mu<\vartheta<\tau\).
Write \(b=\vartheta-\mu>0\), \(A=A_{\vartheta}>0\), and \(w=a^2\).
Then
\[
 F_{\vartheta}(\sqrt w)^2=\frac{(bw+\mu)^2}{Aw-\mu^2},
\]
whose derivative is
\[
 \frac{(bw+\mu)\,[bAw-2b\mu^2-A\mu]}
 {(Aw-\mu^2)^2}.
\]
It therefore decreases and then increases, with at most one turning point,
located at
\[
 w_* = \frac\mu b+\frac{2\mu^2}{A}.
\]
Strict monotonicity in the angle parameter gives
\(F_{\vartheta}(a_0)<F_{\tau}(a_0)=\eta\).  Since \(a\ge a_0\) and
\(F_{\vartheta}(a)\ge\eta\), the point \(a\) must lie on the increasing
branch at or beyond its crossing of the level \(\eta\).  All larger
arguments remain above that level.
\end{proof}

\subsection{Radial sources on the current public domain}

We next connect the scalar inequalities to candidate-independent
contact data. Orthogonality makes the radial source identical on the
public domain for all surviving minimizers, while the current ball
controls the normal produced by contact optimality.

The candidate sphere and its common projection onto the public subspace
are the nested geometry of Criscitiello and
Boumal~\cite[Appendix~A.2, Lemma~26]{CB23}. The distance identity and
radial-direction decomposition below are Lorentz-coordinate
consequences of that geometry. The public domain and the normal reserve
used to preserve contacts are additional parts of our construction.

Let \(V=\mathbb H^d\cap W\) be a public totally geodesic subspace,
put \(U=W^{\perp_{\mathrm L}}\), let \(c\in V\), \(S>0\), and let candidate minimizers be
\begin{equation}\label{eq:cd-candidate-sphere}
 p_v=\cosh S\,c+\sinh S\,v,
 \qquad v\in U,\quad \|v\|=1.
\end{equation}
Let \(P\subseteq V\) be a compact geodesically convex set containing
\(\overline B_V(c,r_S)\).  The set \(P\) may, in particular, be the
geodesic convex hull of all previously installed public balls.

For \(z\in V\), the following quantities are independent of the candidate:
\begin{align}
 D&=d(z,p_v)=\operatorname{arcosh}
       \bigl(\cosh S\cosh d(c,z)\bigr),\label{eq:cd-public-distance}\\
 a&=\frac{\sinh D}{\sinh S}\ge1,\nonumber\\
 u_V&=\frac{\cosh S\,c-\cosh D\,z}{\sinh D}\in T_zV.\nonumber
\end{align}
The unit vector from \(z\) toward \(p_v\) decomposes as
\begin{equation}\label{eq:cd-radial-direction}
 u_v=\frac{p_v-\cosh D\,z}{\sinh D}
     =u_V+\frac1a v,
 \qquad \|u_V\|^2=1-\frac1{a^2}.
\end{equation}
The public restriction of the radial source \(\mu d(\cdot,p_v)\) thus has
gradient \(-\mu u_V\).

\begin{lemma}[The normal reserve supplied by the current ball]
\label[lemma]{lem:cd-normal-reserve}
Let \(z\in P\), and let \(n_V\in T_zV\) satisfy
\[
 \langle n_V,\log_zx\rangle\le0\qquad\text{for all }x\in P.
\]
Write \(\alpha=\|n_V\|\).  Then
\begin{equation}\label{eq:cd-normal-reserve}
 -a\langle n_V,u_V\rangle\ge\eta\alpha.
\end{equation}
Moreover,
\begin{equation}\label{eq:cd-ball-location}
 a^2=1+\coth^2S\,\sinh^2d(c,z),
 \qquad
 d(c,z)<r_S\ \Longleftrightarrow\ a<a_0.
\end{equation}
In this strict-interior case, \(n_V=0\).
\end{lemma}

\begin{proof}
For \(x\ne z\),
\[
 \log_zx=\frac{d(z,x)}{\sinh d(z,x)}
             \bigl(x-\cosh d(z,x)\,z\bigr).
\]
Because \(n_V\) is tangent at \(z\), its support inequality is equivalent
to \(\langle n_V,x\rangle_L\le0\).  Its Lorentz halfspace therefore
contains \(P\), and hence the current public ball.  \Cref{lem:cd-ball-criterion} gives
\[
 -\langle n_V,c\rangle_L\ge\alpha\sinh r_S
                                    =\alpha\eta\tanh S.
\]
Using \eqref{eq:cd-public-distance} and tangency at \(z\),
\[
 -a\langle n_V,u_V\rangle
 =-\coth S\,\langle n_V,c\rangle_L
 \ge\eta\alpha.
\]
The hyperbolic right-triangle identity in
\eqref{eq:cd-public-distance} also gives
\(\sinh^2D=\sinh^2S+\cosh^2S\sinh^2d(c,z)\), proving
\eqref{eq:cd-ball-location}.  A point strictly inside the current ball is
in the relative interior of \(P\), so its intrinsic normal cone is zero.
The zero-dimensional public subspace causes no exception: its tangent
space itself is zero.
\end{proof}

\begin{lemma}[Public contact optimality]\label[lemma]{lem:cd-public-contact}
For \(y\notin P\), choose a minimizer
\[
 z\in\operatorname*{argmin}_{x\in P}
       \{\mu D(x)+d(y,x)\},
 \qquad D(x)=\operatorname{arcosh}
       \bigl(\cosh S\cosh d(c,x)\bigr).
\]
Put \(q=\log_zy/d(z,y)\), and decompose
\[
 q=q_V+\beta e,\qquad q_V\in T_zV,\quad
 \beta=\|q-q_V\|\ge0.
\]
When \(\beta>0\), \(e\in U\) is unit.  Then
\begin{equation}\label{eq:cd-contact-normal}
 n_V=q_V+\mu u_V\in N_P^V(z),
\end{equation}
where \(N_P^V(z)\) denotes the intrinsic normal cone used in
\Cref{lem:cd-normal-reserve}.

For an on-domain query \(y\in P\), one can instead take \(z=y\), choose
any unit \(e\in U\), and set
\[
 q_V=-\mu u_V,\qquad
 \beta=\sqrt{1-\mu^2\|u_V\|^2},\qquad
 q=q_V+\beta e.
\]
Then \(\|q\|=1\), \(\beta>0\), and \(n_V=0\in N_P^V(z)\).
\end{lemma}

\begin{proof}
Compactness gives the minimizing contact.  Since \(y\notin P\), the
distance term is smooth at that contact, with public gradient \(-q_V\).
The other term has gradient \(-\mu u_V\).  First-order optimality for
the convex contact objective over \(P\) gives
\(q_V+\mu u_V\in N_P^V(z)\).  The on-domain formulas give a unit vector
and the zero normal directly; \(\beta>0\) follows from \(\mu<1\).
\end{proof}

\subsection{A fixed residual latitude preserves all archived balls}

The normal reserve now pays for a single residual cut. The next
proposition allows the contact angle to depend on the query while
preserving the scalar inequality at every later stage.

\begin{proposition}[Small-latitude contact routing]
\label[proposition]{prop:cd-small-latitude}
Use the public data \(z,q_V,\beta,n_V\) of
\Cref{lem:cd-public-contact}, and write \(\alpha=\|n_V\|\).
If \(\beta>0\), assume \(\dim U\ge2\) and retain the candidates
\begin{equation}\label{eq:cd-latitude}
 \langle e,v\rangle=-\eta.
\end{equation}
If \(\beta=0\), retain every candidate, without imposing a cut.
For every retained candidate define
\[
 \vartheta=-\langle q,u_v\rangle,
 \qquad n_v=q+\mu u_v\in T_z\mathbb H^d.
\]
Then \(\mu<\vartheta\le1\), and \(\vartheta\) is independent of the
retained candidate.  The halfspace
\[
 H_v^- =\{x\in\mathbb H^d:\langle n_v,x\rangle_L\le0\}
\]
contains \(P\), contains \(p_v\) strictly, and has \(z\) on its boundary.
In addition,
\begin{equation}\label{eq:cd-all-future-scalar}
 F_{\vartheta}(a')\ge\eta\qquad\text{for every }a'\ge a.
\end{equation}
\end{proposition}

\begin{proof}
For \(\beta>0\), the cut gives
\begin{align}
 \vartheta
 &=\mu\left(1-\frac1{a^2}\right)
       -\langle n_V,u_V\rangle+\frac{\beta\eta}{a},
       \label{eq:cd-angle-formula}\\
 B_{\vartheta}(a)
 &=-a\langle n_V,u_V\rangle+\beta\eta
       \ge\eta(\alpha+\beta),\label{eq:cd-B-bound}\\
 A_{\vartheta}-\frac{\mu^2}{a^2}
 &=\alpha^2+\beta^2-\frac{2\mu\beta\eta}{a}
       \le(\alpha+\beta)^2.\label{eq:cd-A-identity}
\end{align}
To check the last identity, expand
\(\|q_V+\mu u_V\|^2+\beta^2\), use
\(\|q_V\|^2+\beta^2=1\), and substitute
\(\langle q_V,u_V\rangle=-\vartheta+\beta\eta/a\).

The triangle inequality gives the useful uniform estimate
\[
 \alpha+\beta
 \ge\sqrt{1-\beta^2}+\beta-\mu\|u_V\|
 \ge1-\mu.
\]
It follows that
\begin{align}
 \vartheta-\mu
 &=\frac{B_{\vartheta}(a)-\mu/a}{a}\nonumber\\
 &\ge\frac{\eta(1-\mu)-\mu/a}{a}
 \ge\frac{\mu(1-\mu)}{(1+\mu)a}>0.
 \label{eq:cd-angle-positive}
\end{align}
The upper bound \(\vartheta\le1\) follows because \(q\) and \(u_v\)
are unit.  In particular, all denominators are positive by
\eqref{eq:cd-positive-denominator}.  Equations
\eqref{eq:cd-B-bound}--\eqref{eq:cd-A-identity} imply
\(F_{\vartheta}(a)\ge\eta\).

If \(a\ge a_0\), \Cref{lem:cd-right-branch} now proves
\eqref{eq:cd-all-future-scalar}.  Suppose instead \(a<a_0\).  By
\Cref{lem:cd-normal-reserve}, \(n_V=0\), so
\(q_V=-\mu u_V\) and
\(\beta=\sqrt{1-\mu^2(1-a^{-2})}\).  Put \(s=a^{-2}\), which belongs
to \((\kappa^2,1]\).  Formula~\eqref{eq:cd-angle-formula} becomes
\[
 \vartheta=H(s):=\mu(1-s)+
 \eta\sqrt{s(1-\mu^2+\mu^2s)}.
\]
For \(0<s\le1\),
\[
 H'(s)=-\mu+
 \frac{\mu(1-\mu^2+2\mu^2s)}
 {(1-\mu^2)\sqrt{s(1-\mu^2+\mu^2s)}}\ge0.
\]
Indeed the square root is at most one and
\(1-\mu^2+2\mu^2s\ge1-\mu^2\).  Since
\(1-\kappa^2=\tau^2\) and \(\eta\kappa=\tau\),
\begin{align*}
 \vartheta&\ge H(\kappa^2)
   =\mu\tau^2+\tau\sqrt{1-\mu^2\tau^2}\\
 &\ge\mu\tau^2+\tau(1-\mu^2\tau^2)
   =\tau+\mu\tau^2(1-\mu\tau)\ge\tau.
\end{align*}
Here \(\sqrt{1-x}\ge1-x\) for \(0\le x\le1\), and
\(\mu\tau<1\).  \Cref{lem:cd-perfect-square} therefore proves
\eqref{eq:cd-all-future-scalar} also in this case.

For \(\beta=0\), no cut is made.  Equations
\eqref{eq:cd-angle-formula}--\eqref{eq:cd-A-identity} hold with every
\(\beta\)-term zero, and \(\alpha\ge1-\mu>0\).  Thus
\(\vartheta>\mu\) by the same estimate
\eqref{eq:cd-angle-positive}, and
\[
 B_{\vartheta}(a)\ge\eta\alpha,
 \qquad A_{\vartheta}-\mu^2/a^2=\alpha^2,
 \qquad F_{\vartheta}(a)\ge\eta.
\]
The nonzero public normal excludes the strict-interior case \(a<a_0\).
\Cref{lem:cd-right-branch} proves
\eqref{eq:cd-all-future-scalar} without imposing any candidate restriction.

Finally, \(n_v\) is tangent at \(z\), so \(z\) is on its hyperplane.
Its projection onto \(T_zV\) is \(n_V\); hence its halfspace contains
\(P\) by the public normal inequality.  Also
\[
 \langle n_v,p_v\rangle_L
 =\sinh D\,\langle q+\mu u_v,u_v\rangle
 =-(\vartheta-\mu)\sinh D<0.
\]
This proves the remaining claims.
\end{proof}

When a cut is made, write
\[
 v=-\eta e+\sigma w,\qquad
 \sigma=\sqrt{1-\eta^2},\qquad
 w\in U\cap e^\perp,\quad\|w\|=1.
\]
It gives the exact recentering of Criscitiello and
Boumal~\cite[Eq.~(7), Lemmas~25--26]{CB23}, expressed here in Lorentz
coordinates with $\eta=\cos\theta$ and $\sigma=\sin\theta$, up to the
orientation of the newly revealed direction:
\begin{equation}\label{eq:cd-recentering}
 \sinh S^+=\sigma\sinh S,
 \qquad
 \cosh S^+\,c^+=\cosh S\,c-\eta\sinh S\,e.
\end{equation}
The enlarged public subspace is
\(V^+=\mathbb H^d\cap\operatorname{span}(W,e)\), and the surviving
minimizers are \(p_w=\cosh S^+c^++\sinh S^+w\).
In particular, the contact direction \(q\) becomes public.

\begin{proposition}[Future public-ball persistence]
\label[proposition]{prop:cd-future-balls}
Fix a contact and a surviving candidate as in
\Cref{prop:cd-small-latitude}.  Suppose that at any later stage
the same physical minimizer admits a representation
\[
 p=\cosh S'\,c'+\sinh S'\,v',\qquad
 c'\in V'=\mathbb H^d\cap W',\quad
 v'\perp_L W',\quad\|v'\|=1,
\]
where \(0<S'\le S\) and the public Lorentz subspace \(W'\) contains
both the old contact point \(z\) and its old direction \(q\).
Then the old halfspace \(H_v^-\) contains the entire intrinsic ball
\(\overline B_{V'}(c',r_{S'})\).
\end{proposition}

\begin{proof}
The quantities \(D=d(z,p)\), \(\vartheta=-\langle q,u_v\rangle\),
and the physical normal \(n_v=q+\mu u_v\) are unchanged.  Put
\(a'=\sinh D/\sinh S'\ge a\).  The Lorentz projection of \(n_v\)
onto \(W'\) is
\[
 n_{W'}=q+\frac{\mu}{\sinh D}
                  (\cosh S'\,c'-\cosh D\,z),
\]
and its orthogonal residual is \((\mu/a')v'\).  Therefore
\begin{equation}\label{eq:cd-future-projected-norm}
 \langle n_{W'},n_{W'}\rangle_L
   =A_{\vartheta}-\frac{\mu^2}{a'^2}>0.
\end{equation}
The identity uses the full Lorentz form on the enlarged public subspace.

Because \(q\in W'\), tangency at \(z\) and the angle identity give
\[
 \langle q,c'\rangle_L
 =-\frac{\vartheta\sinh D}{\cosh S'}.
\]
Also \(\cosh D=\cosh S'\cosh d(z,c')\).  Consequently
\begin{align}
 -\langle n_v,c'\rangle_L
 &=\frac{\vartheta\sinh D}{\cosh S'}
   -\frac{\mu(\sinh^2D-\sinh^2S')}
          {\cosh S'\sinh D}\nonumber\\
 &=\left((\vartheta-\mu)a'+\frac\mu{a'}\right)\tanh S'
   =B_{\vartheta}(a')\tanh S'.
 \label{eq:cd-future-center-pairing}
\end{align}
The right side is positive.  Combining
\eqref{eq:cd-future-projected-norm},
\eqref{eq:cd-future-center-pairing}, and
\eqref{eq:cd-all-future-scalar} yields
\[
 \frac{-\langle n_{W'},c'\rangle_L}
 {\sqrt{\langle n_{W'},n_{W'}\rangle_L}}
 =F_{\vartheta}(a')\tanh S'
 \ge\eta\tanh S'=\sinh r_{S'}.
\]
\Cref{lem:cd-ball-criterion} proves the desired containment.
\end{proof}

\section{A Resisting Oracle for Arbitrary Queries}\label{sec:oracle}

We now turn the geometric statements into a complete oracle lower bound.
All choices below depend only on the current public state and the new query;
no hidden minimizer is fixed until the transcript and the output are known.

\subsection{State, initialization, and one query}

Fix a horizon $T$ and take $d=T+2$. The state consists of a Lorentz public
subspace $W_t$, its totally geodesic subspace $V_t$, a center $c_t\in V_t$,
a residual radius $S_t>0$, and a compact convex public domain $P_t\subset V_t$.
The candidate minimizers are exactly
\begin{equation}\label{eq:state-candidates}
 \mathcal C_t=
 \{\cosh S_t\,c_t+\sinh S_t\,v:
      v\in W_t^{\perp_{\mathrm L}},\ \norm v=1\}.
\end{equation}
The domain $P_t$ is the closed geodesic convex hull of all public balls
installed so far; in particular it contains
$\B_{V_t}(c_t,r_{S_t})$. We retain the archive of distinct queried points
and their selected contacts, vectors, and answers.

The state geometry $(V_t,\mathcal C_t,c_t,S_t)$ corresponds to
$(I_k,S_k,y_k,r_k)$ in Criscitiello and
Boumal~\cite[Appendix~A.2]{CB23}, with $k$ counting actual cuts rather
than queries. We add the public domain $P_t$ and the contact archive to
make this geometry compatible with unrestricted queries. In particular,
the public balls lie in the growing public subspaces and need not be
nested; they are distinct from the residual balls bounded by the
nested candidate spheres in the earlier construction.

Initially,
\[
 W_0=\spanop\{o\},\quad V_0=P_0=\{o\},\quad c_0=o,\quad S_0=R.
\]
The intrinsic ball of any radius in the zero-dimensional space $V_0$ is
$\{o\}$, so the public-ball invariant holds at initialization.

For a query $y$, proceed as follows.
\begin{enumerate}
\item If $y$ was previously queried, return its cached answer and leave
the state unchanged. No cut or new certificate is needed.

\item For a new query outside $P_t$, minimize the candidate-independent
public contact objective
\begin{equation}\label{eq:oracle-contact}
 z\in\operatorname*{argmin}_{x\in P_t}
 \left\{\mu\arcosh\bigl(\cosh S_t\cosh\dist(c_t,x)\bigr)
             +\dist(y,x)\right\}.
\end{equation}
Put $\ell=\dist(y,z)>0$ and $q=\log_z y/\ell$.
For a new query in $P_t$, set $z=y$, $\ell=0$, and choose the unit vector
$q$ by the on-domain construction in \cref{lem:cd-public-contact}.
In both cases write $q=q_V+\beta e$ when $\beta>0$, and put
\[
 D=\arcosh\bigl(\cosh S_t\cosh\dist(c_t,z)\bigr).
\]

\item If $\beta>0$, impose the fixed residual latitude
$\ip{e}{v}=-\eta$ and perform the exact update
\begin{equation}\label{eq:oracle-update}
 \begin{split}
 W_{t+1}&=W_t\oplus\spanop\{e\},\qquad
 V_{t+1}=\HH^d\cap W_{t+1},\\
 \sinh S_{t+1}&=\sigma\sinh S_t,
       \qquad \sigma=\sqrt{1-\eta^2},\\
 \cosh S_{t+1}\,c_{t+1}
       &=\cosh S_t\,c_t-\eta\sinh S_t\,e,\\
 P_{t+1}&=\overline{\gconv}
       \bigl(P_t\cup\B_{V_{t+1}}(c_{t+1},r_{S_{t+1}})\bigr).
 \end{split}
\end{equation}
If $\beta=0$, make no cut and leave $W,V,c,S,P$ unchanged. This case can
occur only for an exterior query in the already public subspace.
The updates of the public subspace, candidate center, and residual radius
are the nested-sphere update of~\cite[Eq.~(7), Lemmas~25--26]{CB23}.
The last line of \cref{eq:oracle-update} additionally retains the convex
hull of all historical public balls, as required for our contact archive.

\item Return and cache the exact proposed answer
\begin{equation}\label{eq:oracle-answer}
 F_y=\mu D+\ell,\qquad
 g_y=\begin{cases}
       -\log_y z/\ell,&\ell>0,\\
       q,&\ell=0.
      \end{cases}
\end{equation}
Archive the contact $(z,q,D)$ and the candidate-independent scalar
$\vartheta=-\ip{q}{u_{z,p}}$ on the retained candidate sphere.
\end{enumerate}
The symbols $t,t+1$ count queries here; on cached queries and zero-cut
queries the state simply repeats. Every minimization is over a compact set,
so it has a solution. Ties can be settled deterministically, for example by
successive coordinate minimization on the compact set of minimizers. The
adversary's computation is not counted as oracle access.

The residual space has dimension at least $d-t\ge2$ before every possible
cut up to the horizon; hence the prescribed latitude is nonempty. In the
on-domain case one can choose a unit residual direction deterministically
from a fixed ambient coordinate basis. The same update then applies.

\subsection{The invariant includes the entire archive}

\begin{proposition}[Finite-prefix consistency]\label[proposition]{prop:prefix-consistency}
After every finite sequence of at most $T$ queries, the following statements
hold.
\begin{enumerate}
\item The candidate sets are nested subsets of the original sphere
$\{p:\dist(o,p)=R\}$, and have the exact representation
\cref{eq:state-candidates}.
\item $P_t$ is compact, convex, public, and contains the current public ball.
Every archived contact point and contact direction belong to the public
Lorentz subspace $W_t$.
\item For every surviving $p$ and every archived contact $i$, the physical
normal
\begin{equation}\label{eq:archive-normal}
 n_{i,p}=q_i+\mu\frac{p-\cosh D_i\,z_i}{\sinh D_i}
\end{equation}
defines a halfspace $H_{i,p}^-$ containing $P_t$ and containing $p$ strictly.
Moreover it contains every future public ball that can arise by continuing
the stated updates, as long as $p$ survives.
\item For every surviving $p$, the single function
\begin{equation}\label{eq:prefix-fixed-objective}
 K_{t,p}=\overline{\gconv}(P_t\cup\{p\}),\qquad
 F_{t,p}(y)=\min_{x\in K_{t,p}}
             \{\mu\dist(x,p)+\dist(y,x)\}
\end{equation}
realizes every archived value and full selected subgradient simultaneously.
\end{enumerate}
\end{proposition}

\begin{proof}
The initial assertions are immediate. For the candidate update, write
$v=-\eta e+\sigma w$, with $w\in W_t^{\perp_{\mathrm L}}\cap e^\perp$
unit, and substitute into \cref{eq:state-candidates}. This gives
\cref{eq:oracle-update} exactly. It changes only the representation and
restricts the index set; it does not move any surviving physical minimizer.
The new Lorentz center is a future-pointing unit timelike vector since
\[
 \ipL{\cosh S_t c_t-\eta\sinh S_t e}
      {\cosh S_t c_t-\eta\sinh S_t e}
 =-\bigl(1+\sigma^2\sinh^2S_t\bigr)
 =-\cosh^2S_{t+1}.
\]
The time orientation is also preserved, because the Lorentz inner product
of the numerator with the future-pointing $c_t$ is $-\cosh S_t<0$.
The public-domain update preserves compactness, convexity, and the archive,
and installs the new ball. When a cut is made, $q_i\in W_{t+1}$ by its
decomposition. For $\beta=0$ it already belongs to $W_t$. All later public
subspaces contain it. Contact points belong to the public domain at the
time of their creation, and thus remain public forever.

Consider a new contact. By \cref{prop:cd-small-latitude}, its halfspace
contains the \emph{whole} old public domain $P_t$ and the retained $p$.
By \cref{prop:cd-future-balls}, it contains the newly installed ball and
every later public ball. For an old contact, the same future-ball
certificate from its creation applies to the newly installed ball: its
old $z_i,q_i$ are still public, its $D_i$ and $\vartheta_i$ are unchanged,
and residual radii only decrease. Since a halfspace is convex, containing
the old public domain and the new ball implies containing their convex
hull. This establishes the archive invariant by induction. Zero-cut and
cached steps do not weaken it.

For clarity, $D_i$ and $\vartheta_i$ are fixed under later updates because
later candidate sets only restrict the same physical points $p$. Already at
the creation of contact $i$, $D_i$ was common to all candidates and the
chosen latitude made $\vartheta_i$ common as well. Thus the physical
normal \cref{eq:archive-normal} is never transported approximately or
replaced by an unrelated vector.

Finally, each $H_{i,p}^-$ contains $P_t$ and $p$, hence contains $K_{t,p}$.
Its normal is therefore in $N_{K_{t,p}}(z_i)$. Apply
\cref{lem:metric-carrier,lem:contact-realization} to the \emph{same}
$K_{t,p}$ for every $i$. This proves simultaneous exact values and selected
subgradients. In particular it proves that the proposed oracle has a legal
completion after every prefix, not just at the terminal time.
\end{proof}

\begin{remark}[Exterior queries remain outside the carrier]\label[remark]{rem:boundary}
All old query directions eventually become public. An external query can
therefore enter the enlarged totally geodesic public span. It still stays
outside the convex carrier $K_{t,p}$: its old supporting normal satisfies
\[
 \ipL{n_{i,p}}{y_i}
 =\sinh\ell_i\,(1-\mu\vartheta_i)>0
 \quad(\ell_i>0),
\]
whereas it is nonpositive throughout $K_{t,p}$. This is the mechanism that
preserves the norm-one external answer. Publicizing a direction does not
force the query into the convex carrier's interior.
\end{remark}

\subsection{Radius, query gaps, and the final output}

Let $k_t\le t$ be the number of actual cuts by time $t$. The exact update
gives the radius recurrence of Criscitiello and
Boumal~\cite[Eq.~(7), Lemma~24]{CB23}:
\begin{equation}\label{eq:radius-budget}
 \sinh S_t=\sigma^{k_t}\sinh R.
\end{equation}
With $T$ from \cref{eq:main-parameters}, for all $t\le T$,
\begin{align*}
 k_t(-\log\sigma)
 &\le T(-\log\sigma)
 \le\log(2\cosh(R/2))\\
 &=\log\frac{\sinh R}{\sinh(R/2)}.
\end{align*}
Thus $S_t\ge R/2$. At a new contact, $D_i\ge S_t$, so every returned
value obeys
\begin{equation}\label{eq:query-gap}
 F_{y_i}=\mu D_i+\ell_i\ge\mu S_t\ge2\eps R.
\end{equation}
Cached values have the same property. The completed functions all have
minimum zero.

Suppose the algorithm stops after $t\le T$ queries and outputs any
$\widehat x\in\HH^d$. At least one residual dimension remains. Choose a
unit $v\in W_t^{\perp_{\mathrm L}}$ such that
\begin{equation}\label{eq:output-choice}
 \ipL{v}{\widehat x}\le0,
\end{equation}
and set $p=\cosh S_t c_t+\sinh S_t v$. Such a choice exists by taking one
of the two signs of any residual direction (or the direction opposite the
residual projection of $\widehat x$). Then
\begin{align*}
 \cosh\dist(\widehat x,p)
 &=-\cosh S_t\ipL{\widehat x}{c_t}
        -\sinh S_t\ipL{\widehat x}{v}\\
 &\ge\cosh S_t\cosh\dist(\widehat x,c_t)
 \ge\cosh S_t.
\end{align*}
Consequently $\dist(\widehat x,p)\ge S_t$. Select the final objective
$F_{t,p}$ from \cref{eq:prefix-fixed-objective}. Its floor gives
\begin{equation}\label{eq:output-gap}
 F_{t,p}(\widehat x)\ge\mu S_t\ge2\eps R.
\end{equation}

This is a valid deterministic resisting-oracle argument with a single fixed
objective. Explicitly, first simulate the deterministic algorithm using the
public replies above. After it stops, choose $p$ by \cref{eq:output-choice}
and fix $F_{t,p}$. Give each queried point its cached selected subgradient;
on unqueried points select the unique minimum-norm element of the
subdifferential, whose existence is proved in \cref{lem:selection}. This defines
one fixed oracle map. Rerunning the deterministic algorithm against this
fixed objective and oracle produces exactly the simulated transcript,
by induction on the queries, and hence the same failed output.

The nontermination alternative also admits a fixed-instance witness.
If the simulation requests query $T+1$ before returning an output, choose
any $p\in\mathcal C_T$ and fix the completion $F_{T,p}$ and its oracle at
that prefix. Rerunning reproduces all first $T$ replies and the same
decision to continue, so the algorithm exceeds the proposed query budget
on a fixed legal instance. If an algorithm instead stalls after any finite
prefix of length $t\le T$, fix that prefix's completion; its same internal state then
also fails to return an output. Thus every algorithm either returns an
inaccurate output within the budget or fails its budget-and-output
guarantee on a fixed instance.

\subsection{Explicit uniform constant}

For $0<\eps\le1/64$ and $\mu=4\eps$,
\[
 \eta^2=\frac{64\eps^2}{(1-16\eps^2)^2}\le65\eps^2<\frac12,
 \qquad
 -\log\sigma=-\frac12\log(1-\eta^2)\le\eta^2.
\]
Also,
\[
 \zeta_R=R+\frac{2R}{e^{2R}-1}\le R+1,
\]
and
\begin{align*}
 \log(2\cosh(R/2))
 &\ge\max\{\log2,R/2\}\\
 &\ge\frac{\log2+R/2}{2}
 \ge\frac{R+1}{4}\ge\frac{\zeta_R}{4}.
\end{align*}
Therefore
\[
 T+1>
 \frac{\log(2\cosh(R/2))}{-\log\sigma}
 \ge\frac{\zeta_R}{260\eps^2}.
\]
Since every algorithm guaranteeing an output after at most $T$ queries
fails on some legal instance, $Q_d^\star(R,\eps)\ge T+1$ for $d=T+2$.
For an algorithm receiving $L$-Lipschitz objectives, run the same public
simulation with each proposed value and vector multiplied by $L$;
the final completion is $L F_{t,p}$. Thus the simulation remains valid
even if the algorithm changes its queries in response to this scaling.
This proves
the lower half of \cref{thm:main}.

\section{Matching Upper Bound}\label{sec:upper}

For completeness, we prove the projected Riemannian subgradient bound
of Zhang and Sra~\cite[Theorem~9]{ZS16} with explicit constants for our
localization convention. The only curvature loss comes from the Hessian
of squared distance along an unprojected step. Projection keeps the
queried iterates in the known ball, and selecting the smallest observed
value gives an executable output without knowing the optimum value.

\begin{proposition}[Projected subgradient upper bound]\label[proposition]{prop:upper}
For all $d\ge2$, $R,L>0$, and $0<\eps<1$,
\[
 Q_d^\star(R,\eps)\le
 \left\lceil\frac{\zeta_{3R}}{\eps^2}\right\rceil
 \le\left\lceil\frac{3\zeta_R}{\eps^2}\right\rceil.
\]
The algorithm needs no knowledge of the minimum value.
\end{proposition}

\begin{proof}
Let $p\in\B(o,R)$ minimize $f$, put
\[
 N=\left\lceil\frac{\zeta_{3R}}{\eps^2}\right\rceil,
 \qquad \alpha=\frac{R}{L\sqrt{\zeta_{3R}N}},
\]
and run
\[
 x_0=o,\qquad y_j=\exp_{x_j}(-\alpha g_j),\qquad
 x_{j+1}=\Pi_{\B(o,R)}(y_j),\qquad 0\le j<N.
\]
Query $(f(x_j),g_j)$ at each $x_j$ and output the queried point with smallest
observed value. Global Lipschitzness and the support inequality imply
$\norm{g_j}\le L$: apply the support in the direction $g_j$ and compare
with the Lipschitz upper bound.

For $h_p(x)=\tfrac12\dist(x,p)^2$, the hyperbolic Hessian has eigenvalue
$1$ in the radial direction and $r\coth r$ in orthogonal directions,
where $r=\dist(x,p)$; at $r=0$ both values are $1$. This follows by
differentiating the polar metric $dr^2+\sinh^2r\,d\theta^2$.
Since $\alpha L\le R$, the unprojected step remains within distance
$3R$ of $p$, so $\nabla^2h_p\preceq\zeta_{3R}I$ along it.
Taylor's bound and subgradient support yield
\begin{align*}
 \dist(y_j,p)^2
 &\le\dist(x_j,p)^2+2\alpha\ip{g_j}{\log_{x_j}p}
                  +\zeta_{3R}\alpha^2L^2\\
 &\le\dist(x_j,p)^2-2\alpha(f(x_j)-f(p))
                  +\zeta_{3R}\alpha^2L^2.
\end{align*}

Projection onto a closed convex set does not increase distance to a point
of that set. Here this also follows directly from the same Hessian formula.
If $z=\Pi_C(y)$ and $p\in C$, minimality along the segment $[z,p]$ gives
$\ip{-\log_z y}{\log_zp}\ge0$. Since
$\nabla^2(\tfrac12\dist(\cdot,y)^2)\succeq I$, integration along that
segment yields
\[
 \dist(y,p)^2\ge\dist(y,z)^2+\dist(z,p)^2.
\]
Thus the preceding recursive inequality also holds with $x_{j+1}$ on the
left. Summing over $j$, using $\dist(o,p)^2\le R^2$, and taking the best
queried value gives
\begin{align*}
 \min_{0\le j<N}(f(x_j)-f(p))
 &\le\frac{R^2}{2\alpha N}+\frac{\zeta_{3R}\alpha L^2}{2}\\
 &=LR\sqrt{\frac{\zeta_{3R}}N}\le\eps LR.
\end{align*}
Finally, $\tanh(3R)\ge\tanh R$, hence $\zeta_{3R}\le3\zeta_R$.
\end{proof}

\section{Discussion}\label{sec:scope}

The theorem identifies the joint curvature--accuracy cost in the
dimension-uniform deterministic model. Its lower bound matches the
established projected-subgradient rate, including the curvature factor,
throughout a universal small-accuracy range. The proof also explains a
structural source of difficulty: a global objective can retain an entire
archive of exact first-order information while its possible minimizers
lose only $O(\eps^2)$ of a geometric potential per query.

Several boundaries of the characterization are substantive. The hard
dimension is $d=T+2$, so the theorem does not prescribe the best rate at
each fixed dimension. Combining the available algorithms gives the
dimension-sensitive upper envelope
\[
 Q_d^\star(R,\eps)
 \le O\!\left(\min\left\{
 \zeta_R\eps^{-2},\ d^2[R+\log(e/\eps)]
 \right\}\right),\qquad 0<\eps<1,
\]
using \Cref{prop:upper} and the Klein cutting-plane guarantee~\cite{Klein26}.
The present lower bound determines the supremum over dimensions, rather
than the full dependence on $d$. It also uses deterministic replay;
extending the conclusion to randomized algorithms would require an
argument controlling their randomness on fixed instances. The choice
of a final candidate after a deterministic transcript does not itself
provide such an extension.

The function class assumes global geodesic convexity and global
Lipschitz continuity, with exact ordinary first-order access. Additional
smoothness, horospherical support structure, or stronger oracles lead
to different questions. Our persistence calculation also exploits
constant-curvature hyperbolic identities, so it does not establish the
same construction on every Hadamard manifold. Finally, the lower-bound
constants are proved for $0<\eps\le1/64$. For $\eps\ge1$, the known
reference point already satisfies the target without a query because
$f(o)-\min f\le LR$. We do not optimize the transition between these
accuracy regimes.

The result completes the unrestricted Lipschitz case of the program
initiated by Criscitiello and Boumal~\cite{CB23}. Their problem
formulation, conjectured rate, and nested hyperbolic geometry provide
the foundation; the carrier and persistence arguments above supply a
way to realize that geometry with exact oracle consistency for arbitrary
queries. Criscitiello~\cite{C26} independently reaches the same
complexity characterization through the different construction discussed
in \Cref{sec:related}.

\paragraph{Acknowledgments.}
We thank Christopher Criscitiello for bringing the precise correspondence
with the geometric construction in Criscitiello and Boumal~\cite{CB23}
to our attention and for helpful correspondence about the independent
concurrent results.

\end{document}